\documentclass{amsart}
\usepackage{amssymb}
\usepackage{amsthm}
\usepackage{pb-diagram}
\usepackage{color}

\usepackage[textsize=tiny]{todonotes}

\newtheorem{theorem}{Theorem}[section]

\newtheorem{claim}[theorem]{Claim}
\newtheorem{fact}[theorem]{Fact}

\newtheorem{proposition}[theorem]{Proposition}
\newtheorem{lemma}[theorem]{Lemma}
\newtheorem{corollary}[theorem]{Corollary}

\newtheorem{question}[theorem]{Question}

\newtheorem*{question*}{Question}

\theoremstyle{definition}

\newtheorem{question+}[theorem]{Question}

\newtheorem{defn}[theorem]{Definition}

\theoremstyle{remark}

\newtheorem{notation}{Notation}

 \DeclareMathOperator{\intr}{int}
  \DeclareMathOperator{\bd}{bd}

\newcommand{\WR}{\widetilde{\cal R}}

\newcommand{\la}{\langle}
\newcommand{\ra}{\rangle}

\newcommand{\sub}{\subseteq}

\newcommand{\ldim}{\ensuremath{\textup{ldim}}}

\newcommand{\cal}[1]{\ensuremath{\mathcal{#1}}}
\newcommand{\Cal}[1]{\ensuremath{\mathcal{#1}}}

\newcommand{\cl}[1]{\begin{overline}{#1}\end{overline}}

\newcommand{\es}{\ensuremath{\emptyset}}

\newcommand{\dom}{\ensuremath{\textup{dom}}}

\newcommand{\sm}{\setminus}

\newcommand{\Z}{\mathbb{Z}}
\newcommand{\N}{\mathbb{N}}
\newcommand{\Q}{\mathbb{Q}}
\newcommand{\R}{\mathbb{R}}

\title[Structure theorem for i-minimal expansions of the real additive group]
{Structure theorem for i-minimal expansions of the real additive ordered group}

\begin{document}

\author {Alex Savatovsky}

\address{Department of Mathematics and Statistics, University of Konstanz, Box 216, 78457 Konstanz, Germany}

\email{alex.savatovsky@uni-konstanz.de}


\date{\today}
\maketitle
\begin{abstract}
In this paper, we prove that for an o-minimal expansion of the real additive group $\cal R$ and a set $P\subseteq \R$  of dimension $0$ such that $\la\cal R,P\ra$ is sparse,  has definable choice and every definable set has interior or is nowhere dense then, for every definable set $X$, there is a family $\{X_t:\; t\in A\}$ definable in \Cal R and a set $S\subseteq A$  of dimension $0$ such that $X=\bigcup_{t\in S}X_t$. Moreover, in the d-minimal setting, there is a finite decomposition of $X$ into sets of the previous form such that for every $t\in S$ $X_t$ is relatively open in $\bigcup_{t\in S}X_t$. 
\end{abstract}
\section{introduction}
I-minimal structures have been introduced by Miller in \cite{Mil1} as the expansions of the real ordered line that does not define some dense set without interior (although not named at the time). The study of such structures is part of the bigger program of studying tameness in expansions of o-minimal structures (see \cite{Mil1}). 
\par D-minimal structures are a special case of i-minimal structures and have been first studied by van den Dries in \cite{VDD1} ('83) while proving a quantifer elimination result for the structure $\la\cl{\R},2^\Z\ra$. Since then,  many structures have been proven to be d-minimal and all of them are reducts of structures of the form $\la\cal R,P\ra^\#$ where $\cal R$ is an o-minimal expansion of the real ordered additive group, $P\subseteq \R$ is a set of dimension $0$ and $\la\cal R,P\ra^\#$ denotes the expansion of $\la\Cal R,P\ra$ by predicates for any subset of $P^k$ for any $k\in \N$ (for examples see \cite{Mil1}, \cite{MT}, \cite{FM2}, \cite{ToVo}).  
\smallskip
\par \textbf{For the rest of this paper, we fix \cal R, an o-minimal expansion of the real additive ordered group and  a set $P\subseteq \R$ of dimension $0$. We denote by $\WR=\la\Cal R,P\ra$.}
\smallskip
\begin{defn}
We denote by $\WR^\#$ the structure $\WR$ with a predicate for any subset of $P^k$ for any $k\in \N$.
\par We say that $\WR$ is {\em i-minimal} if every set definable in $\WR$ has interior or is nowhere dense. This denomination has been created by A. Forniasiero (see \cite{For2}) but note that C. Miller is calling it {\em noiseless}.
\par we say that $\WR$ is {\em sparse} if for every function $f:\R^k\to \R$ definable in $\cal R$, $f(P^k)$ has no interior. 
\par We say that $\WR$ is {\em d-minimal} if given a family of sets of topological-dimension $0$, there is a uniform finite decomposition of its elements into discrete sets. It is easy to see that d-minimal structures are i-minimal and sparse.
\par Let $X$ be a set definable in $\WR$. We say that  a definable family $\{X_t:\; t\in S\}$ is an 
\cal R-decomposition of $X$ if:
\begin{enumerate}
\item $X=\bigsqcup_{t\in S}X_t$
\item there is a  family of o-minimal cells $\{X_t:\; t\in A\}$ definable in \Cal R such that $S\subseteq A$
\item $S$ has topological-dimension $0$.
\end{enumerate}
\par A definable set $X$ is called an {\em \cal R-embedded manifold} (\Cal R-EM) if there is $\{X_t:\; t\in S\}$, an \cal R-decomposition of $X$ such that: \begin{enumerate}
\item there is a projection $\pi$, for every $t\in  S $,  $X_t$ is the graph of a continuous function over $\pi(X_t)$ and $\pi(X_t)$ is open and connected.
\item For every $t,t'\in S$, either $\pi(X_t)=\pi(X_{t'})$ or $\pi(X_t)\cap \pi(X_t')=\es$.
\end{enumerate}
As a convention, when talking about an \cal R-deomposition of an \cal R-EM, we always mean that this decomposition satisfies $1$ and $2$.
\par We say that $\WR$ has {\em the decomposition into finitely many \cal R-EM property} (\cal R-DEM) if every set definable in $\WR$ admits a decomposition into finitely many \Cal R-EM. 
\end{defn}
\par \textbf{For the rest of this paper, we assume that $\WR$ is i-minimal, sparse and has definable choice.}
\smallskip
The main theorem of this paper is the following:
\begin{theorem}\label{main}
$\WR^\#$ has \cal R-DEM.
\end{theorem}
As an easy consequence we establish a stronger version of Theorem \ref{main} for d-minimal structures.
\begin{theorem}\label{main2}
If $\WR$ is d-minimal then for every definable set $X$, there is a decomposition of $X$ into finitely many \cal R-EM $X_1,\ldots,X_n$ such that for every $i$ there is $\{X_t:\, t\in S\}$, some \Cal R-decomposition of  $X_i$ such that for every $t\in S$, $X_t$ is relatively open in $X_i$. 
\end{theorem}
\par Theorem \ref{main2} can be seen both as a generalization and a reduction of the decomposition into special manifold of Thamrongthanyalak (\cite[Theorem B]{Tam}). A generalization because it does not assume anymore that \cal R expands the real field and a reduction because it only concerns the expansions of an o-minimal structure by a unary set of dimension $0$. 
Anyway, having the uniform definability of the connected components of the special manifolds plays a crucial role in some applications such as:
\begin{theorem}\label{app1}
Let $\WR$ be d-minimal. If $f$ is a smooth  function definable in $\WR^\#$, over a domain that is definable in \Cal R then $f$ is definable in $\cal R$. 
\end{theorem} 

\smallskip 
\par {\em Structure of the paper:} In the first part, we establish some notations, definitions and basic results. In the second part we prove Theorem \ref{main} and \ref{main2}. In the third part, we discuss some applications (such as Theorem \ref{app1}) and some related questions. 
\smallskip 
\par {\em Aknowledgement: } We would like to thank  Pantelis Eleftheriou, Chris Miller and Philipp Hieronymi for usefull discussions. We would like to thank particularily Erik Walsberg  for usefull discussion in general and for pointing some possible applications. 
\section{Preliminaries}  By a cell, we mean an o-minimal cell definable in $\cal R$.  If $S\sub \R^n$ is a set, its closure is denoted by $\cl S$, with sole exception  $\overline \R$, which denotes the real field. We denote by $\intr (S)$ the interior of $S$ and by $\bd(S)$ its boundary, that is $S\sm \intr(S)$. By an open box $B\subseteq \R^n$, we mean a set of the form
$$B=(a_1, b_1)\times \ldots \times (a_n, b_n),$$
for some $a_i< b_i\in \R\cup\{\pm\infty\}$. By an open set we always mean a non-empty open set. $for a,b\in \R^n$ we denote by $[a,b]$ the set $\{xa+(1-x)b:\; x\in [0,1]\}$. For two sets $A,B\subseteq \R^n$ we denote by $[A,B] $ the set $\bigcup_{a\in A,\, b\in B}[a,b]$. For $x\in \R^n$ and $\delta\in \R$, we denote by $\cal B(x,\delta)$ the box of center $x$ and of radius $\delta$. We denote by $\cal B(\delta)=\cal B(0,\delta)$. 
 By `definable', we mean definable in $\WR$ with parameters. \\

For any set $X\sub R^n$, we define its \emph{dimension} as the maximum $k$ such that some projection of $X$ to $k$ coordinates has non-empty interior. We call a set {\em small} if it has dimension $0$ and we call a family small if it is indexed by a small set.
\par We assume to be known the basics of o-minimality as they can be found in \cite{VDD}.
We also assume to be known the basics of d-minimality as they can be found in \cite{Mil1}.

\subsection{sparseness}
In this subsection  we recall some central result about sparse structures and we establish some property of the dimension in our setting.

\par The following theorem from Friedman and Miller is a central result for the main results of this paper.  
\begin{fact}\label{FM}(see \cite{FM2} last claim of the proof of Theorem A)  
Let $\la \cal R,P'\ra $ be sparse. Let $A\subseteq \R^{n+1}$ be definable in $\la \cal R,P'\ra^\#$ such that for every $x\in \R^n$, $A_x$ has no interior. Then there is a function $f:\R^{m+n}\rightarrow \R$, definable in \Cal R, such that for every $x\in \R^n$, $$A_x\subseteq \cl{f(P'^m\times\{x\})}.$$
\end{fact}
\begin{lemma}\label{DIM}
Let $\{X_t:\, t\in S\}$ be a definable small subfamily of a family definable in \cal R. Then $$\dim(X=\bigcup_{t\in S}X_t)=\max_{t\in S}(\dim(X_t)).$$ 
\end{lemma}
\begin{proof}
First of all, if $\dim(X_t)$ is equal to the dimension of the ambient space for some $t\in S$, the result is obvious. Thus, we may assume that it is not the case. 
Since $\{X_t:\; t\in S\}$ is a subfamily of a family definable in \Cal R, by uniform cell-decomposition in o-minimal structures, we may assume that there is some projection $\pi$, such that for every $t\in S$, $\pi(X_t)$ is open, connected and $X_t$ is the graph of a continuous function over $\pi(X_t)$.  By sparseness there is a function $f$, definable in \Cal R, such that  for  every $z\in \pi(X)$, $\pi^{-1}(z)\cap X\subseteq \cl{f(z,P^k)}$ and by i-minimality, this last set has dimension $0$. Thus, we have the result. 
\end{proof}
\begin{fact}\label{pi good} Let $X\subseteq \R^m$ be a definable set of dimension $n$. Then, for any projection $\pi$ onto some $n$-coordinates, $$\dim\big(\{x\in \pi(X):\; \dim(\pi^{-1}(x)\cap X)>0 \}\big)<n.$$
\end{fact}
\subsection{The Hausdorff topology}
\begin{defn}
We denote by $\cal K(\R^n)$ the collection of compact subsets of $\R^n$. 
\par The {\em Hausdorff distance} on $\cal K(\R^n)$ is defined by $$d_H(X,Y)=\sup\{d(x,Y),d(X,y):\; x\in X,y\in Y\}.$$
It is easy to see that $d_H$ is a metric and it gives us a topology on $\cal K( \R^n)$, the {\em Hausdorff topology}. 
 Let $(X_n)$ be a sequence of compact sets which is converging in the Hausdorff topology. We denote by $\lim_H(X_n)$ its limit. 
 For a family of compact sets $\cal C$, we denote by $\cl{\cal C}^H$ the closure of $\cal C$ in the Hausdorff topology.
\end{defn}
The following theorem is the technical key to prove $\cal R$-DEM. 
\begin{theorem}\label{Hausdorff fam} (see \cite{VDD Hausdorff} for an overview of the results on this topic)
Let $\cal C$ be a family of compact sets, definable in \Cal R and contained in $\R^n$. Then $$\{Y\in \cal K(\R^n):\; Y\in \cl{\cal C}^H\}$$ is definable in \Cal R.  
\end{theorem}

\begin{fact} \label{Cynk}(see \cite[Theorem 3]{cynk})
Let $Y_{j}$ be a sequence in $ \cal K(\R^n)$ and $A\in \cal K(\R^n)$ such that $\bigcup_jY_j$ is contained in a compact set. Then $(Y_j) $ converge to $A$ in the Hausdorff topology if and only if the following conditions are satisfied:
\par 1) Let $y_i\in Y_i $. If $(y_i)$ converge to $y$ then $y\in A$.
\par 2) If $y\in A$ then there is a sequence $(y_i)_i$ with $y_i\in Y_i$ such that $\lim_iy_i=y$.  
\end{fact}
\begin{fact}\label{Dini}(Dini's Theorem, see for example \cite[Theorem 7.13]{Rudin})
Let $\{f_n\}$ be a monotically increasing  sequence (that is that for every $n$, $ \dom(f_n)=\dom(f_{n+1})$ and for every $x\in \dom(f_n) $ $f_n(x)<f_{n+1}{(x)}$) of continuous functions over a compact domain which converges pointwise to a continuous function then the convergence is uniform. The result holds, of course, for  a decreasing sequence satisfying the same conditions.
\end{fact}
\begin{corollary}\label{Hlim=cl}
Let $\{Y_j:\; j\in \N \}$ be a small subfamily of a  family of graphs of continuous functions  $\{f_j:Z\to \R:\; j\in \N\}$ definable in \cal R such that:
\begin{itemize}
\item $Z$ is an open cell.
\item For every $i<j$, for every $x\in Z$, $f_i(x)>f_j(x)$.
\item $\bigcup_jY_j$ is contained in a compact set.
\end{itemize}
Let $A=\{\lim_j(z,f_j(z)):\; z\in Z\}$. We assume that $A$ is  the graph of a continuous function $f:Z\rightarrow \R$.  Then $$\lim_H(\{\cl{Y_j}\})\cap \pi^{-1}(Z)=A.$$
\end{corollary}
\begin{proof}
First of all, by definition $A$ satisfies the second condition of Fact \ref{Cynk}. 

\par Let $y_j\in Y_j$ such that $(y_j)$ converges to $y\in \pi^{-1}(Z)$. There is a bounded and closed box $B\subseteq Z$ and $N\in \N$ such that for every $j>N$, $\pi(y_j)\in B$. 
By Fact \ref{Dini}, the sequence $\{f_j{\upharpoonright B}:\; j\in \N\}$ converges uniformly to $f_{\upharpoonright B}$ and thus, $y\in A$. 
 \end{proof}
\subsection{On semibounded structures}
In this subsection we present some properties of semibounded structures. See \cite{mpp} for the basis on semibounded structures. 
\begin{defn} 
We say that a continuous function definable in \Cal R is a \emph{pole} if it is a bijection between a bounded and an unbounded interval. We say that \cal R is {\em semibounded} if it does not define a pole.
\par We say that a set $X$ is a \emph{cone} if   $$X=B+\sum_{a\in (\R^{>0})^l}\sum_i a_iv_i:=\{x\in \R^n\,:\; \exists b\in B,\, a\in (\R^{>0})^l,\; x=b+\sum_ia_iv_i\}$$
where $B$ is a bounded definable set, $v_1,\ldots,v_l$ are linearly independent vectors over $\R$ and for every $x\in X$, there are unique $b\in B$ and $a \in (\R^{>0})^l$ with $x=b+\sum_ia_iv_i$. 
\end{defn}
\begin{fact}\label{Ed}(see \cite[Fact 1.6]{ed-sbd})
The following are equivalent:
\begin{enumerate}
\item \cal R is semibounded
\item Every definable function from $\R$ to $\R$ is ultimately affine.
\item There is no definable real closed field with domain $\R$ such that the ordering agrees with $<$.
\item  Every definable set can be decomposed into finitely many cones. 
\end{enumerate} 
\end{fact}

\smallskip
\par In the semibounded setting, there is another  notion of dimension that characterizes the degree of unboundedness of a definable set.
\begin{defn} Let $B$ and $\{v_1,\ldots,v_l\}$ be as in  Fact \ref{Ed} and 
 $X$ be the  cone  $B+\sum_{a\in (\R^{>0})^l}a_iv_i$. We define the \emph{long-dimension of} $X$ ($\ldim(X)$) to be $l$. For $X$ a semibounded set, we define the long dimension of $X $ to be the maximal long dimension of a cone contained in $X$. 
 We remark that this notion of dimension is preserved under  affine bijections. 
\end{defn}
\begin{fact}\label{ldim}(see \cite[Lemma 3.6]{el-sbd})
For $\{X_t\,:\; t\in A\}$ a semibounded family  such that  $A$ is bounded, we have $$\ldim(\bigcup_{t\in A}X_t)=\max\big(\ldim(X_t),\, t\in A\big).$$
\end{fact}
\begin{proposition}\label{lm sb}
Let $\{X_t\subseteq \R^n:\; t\in A\subseteq \R^m\}$ be a semibounded family. Then there are finitely many semibounded families $\{B_{i,t}\subseteq \R^n:\; t\in C_i\}$, $D_i\subseteq \R^n$ some bounded sets, $V_i$ some vector subspace of $\R^n$ and $\{a_{i,t}:\; t\in A\}$ such that:
\begin{enumerate}
\item For every $i$ $C_i$ and $D_i$ are bounded.
\item For every $t\in A$, for every $i$ there is $t_i\in C_i$ with $X_t\subseteq \bigcup_i B_{i,t_i}+V_i+a_{i,t}$ and has the same dimension.
\item For every $t\in C_i$ $B_{i,t}\subseteq D_i$.
\item For every $t\in A$, $x\in B_{i,t}+V_i+a_{i,t}$ there are unique $b\in B_{i,t}$ and $v\in V_i$ such that $x=b+v+a_t$.
\end{enumerate} 

\end{proposition}
\begin{proof}
First, by uniform cell decomposition, we may assume that for every $t\in A$, $X_t$ is a cell of dimension $k$.
Let $X=\bigcup_{t\in A}X_t\times\{t\}$. By Fact \ref{Ed} we may assume that $X$ is a cone of the form $B+\sum_{a\in (\R^{>0})^l}\sum_i a_iv_i$ where $B$ is a bounded definable set, the $v_i$'s are independent vectors in $\R^{n+m}$ and $(*):$ for every $x\in X$ there are unique $b\in B$ and $a\in (\R^{>0})^l$ such that $x=b+av$.
\par Let $V=\la v_1,\ldots,v_l\ra$ and let $$Y=B+V.$$
It is clear that $X\subseteq Y$.
We first show that   $Y$ has $(**):$ for every $x\in Y$, there are unique $b\in B,\,a\in \R^l$ such that $x=b+\sum_ia_iv_i$. Suppose there are $x\in Y$, $b_1,b_2\in B$, $a_1,\,a_2\in \R^l$ such that $$x=b_1+a_1v=b_2+a_2v.$$
Let $a=(\max(|(a_1)_i|,|(a_2)_i|)+1)_i$. We have that $$b_1+(a_1+a)v=b_2+(a_2+a)v=x+av\in X$$ and by $(*)$, $b_1=b_2$ and $a_1=a_2$. 
\par We show that $\dim(X)=\dim(Y)$.   By $(*)$, $\dim(X)=\dim(B)+l$ and by $(**)$, $\dim(Y)=\dim(B)+l$, thus we have the result. For $t\in A$, let $Y_t=\R^n\times\{t\}\cap Y$. Let $\pi_2$ be the projection on the last $m$ coordinates and let $E=\pi_2(Y)$. For every $t\in A$, we show that $\dim(X_t)=\dim(Y_t)$. Without loss of generality, we may assume that for every $t,t'\in A$, $\dim(Y_t)=\dim(Y_{t'})=s$.  Since each $X_t$ is a cell of dimension $k$, $$\dim(X)=k+\dim(A)=\dim(Y)=s+\dim(E).$$ Moreover, since $X\subseteq Y$, for every $t\in A$, $$X_t\subseteq Y_t \text{ and } A\subseteq E.$$ Thus $$\dim(X_t)\leq \dim(Y_t)\text{ and }\dim(A)\leq \dim(E)$$ and we have the equality. 
\par We define an equivalence relation $\sim $ on
$\R^{m}$:
$$x\sim y \text{ if and only if there is $v\in V$ such that }\R^n\times\{x\}+v=\R^n\times\{y\}.$$ Let $V_0=\R^n\times\{0\}\cap V$ and let $V_1$ be a supplementary of $V_0$ in $V$: $V_0\oplus V_1=V$.  
 From $(**)$, it is easy to see that for every $t\sim t'$ there is a unique $v\in V_1$ such that $$\R^n\times\{t\}+v=\R^n\times \{t'\}.$$ Let $\gamma:\pi_2(B)\to \pi_2(B)$ be a choice function for $\sim$ and for $t\in \R^m$ let $B_t=B\cap \R^n\times\{t\}$. Let $$f:\bigcup_{t\in \pi_2(B)}B_t\to \R^{n+m},$$ $$x\in B_t \mapsto x+v \text{ where $v\in V_1$ is the unique such that $x+v\in B_{\gamma(t)}$}.$$ 
 Note that $f$ is injective by $(**)$. For $t\in \R^m$, let $B_{\cl{t}}=\bigcup_{t'\sim t} f(B_{t'})$. Note that $B_{\cl{t}}\subseteq \R^n\times\{\gamma(t)\}$. 
 We finish the proof with the following claim:
 \begin{claim}
 For every $t\in R^m$, $Y_t\times\{t\}=B_{\cl{t}}+V_0+a_t$ where $a_t$ is the unique element in $V_1$ such that $\R^n\times\{t\}=\R^n\times\{\gamma(t)\}+a_t$.
 \end{claim}   
 \begin{proof}
 Let $(x,t)\in Y_t\times\{t\}$. By $(**)$, there are unique $b\in B$ and $v\in V$ such that $(x,t)=b+v$. There are unique $v_0\in V_0$, $v_1,v'_1\in V_1$ such that $v'_1+f(b)=b$ and $v=v_0+v_1$. Thus 
 $$(x,t)=b+v=f(b)+v_0+v_1+v'_1.$$
 The other inclusion is easy and we have the result.
 \end{proof}
It is then easy to check that $\{\pi_1(B_{\cl{t}}):\; t\in \pi_2(B)\}$, $\pi_1(V_0)$ and $\{\pi_1(a_t):\; t\in A\}$ satisfy  the properties $(1)-(5)$.
\end{proof}
\section{The proof}
\begin{theorem}\label{L-DEM}
Let $X=\{X_x\subseteq \R^m:\; x\in A\subseteq \R^p\} $ be a  family of sets of dimension $n$ definable in $\WR^\#$. Then there is a uniform decomposition of the  $X_x$'s into finitely many $\cal R$-embedded manifolds.
\end{theorem}

\begin{proof}
We are doing a double induction  on $n\leq m$. For $m=0$ or for  any $m$ and $n=0$, the result is obvious. We assume that  the result holds for $n$ and $m-1$. We also assume that the result holds for $n-1<m-1$ and $m$ and prove it for $n$ and $m$. We  deal with the case $n<m-1$, $n=m-1$ and $n=m$ separately.
\par   We first prove the following claims:
\begin{claim}\label{cl10}
Let $Z$ be a definable set of dimension $n$ and let $ \{ Z_t:\, t\in S\}$ be an \Cal R-decomposition of $Z$. Then there is a  decomposition of $Z$ into finitely many \Cal R-EM.
\end{claim}
 \begin{proof}  By definable choice, there is $\gamma:Z\to S$, $x\mapsto t$ such that $x\in Z_t$. For $t\in S$ let $Z'_t=\gamma^{-1}(t)$. Note that the $Z'_t$'s are no longer definable in \cal R but $Z'_t\subseteq Z_t$. By uniform cell decomposition, we may assume that the $Z_t$'s are cells and thus are homeomorphic via $\pi$ to some open cell in $\R^n$. By induction, there is an uniform decomposition of the $Z'_t$'s into \cal R-EM. For $t\in S$, let $Z''_t$ be the union of the \cal R-EM of dimension $n$ in its decomposition. For $t\in S$ let $$W_t=\bd(\pi(Z''_t))\text{ and let }W=\cl{\bigcup_{t\in S}W_t}.$$ Note that, by Lemma \ref{DIM}, $\dim(W)<n$. By induction there is a decomposition of $\R^n\setminus W$ into \cal R-EM $W_1,\ldots,W_s$ of the form $W_i=\bigcup_{t\in S_i}W_{i,t}$, where there is  family $\{W_{i,t}:\; t\in A_i\}$, definable in \Cal R, such that $S_i\subseteq A_i$ and $S_i$ is small. 
 \par Let $i$ such that $\dim(W_i)=n$ and let $t\in S_i$. Observe that, by definition of $W$, 
 $\pi^{-1}(W_{i,t})\cap Z$ is composed of connected components which are graphs of continuous functions over $W_{i,t}$. 
  Thus, if $W_1,\ldots,W_s$ are the \cal R-EM of dimension $n$ in the decomposition of $\R^n\setminus W$, $ \pi^{-1}(W_{i})\cap Z$  is an \cal R-EM. 
  \par By definition of  $W_1,\ldots, W_s$, $\dim(\R^n \sm \bigcup_{i\leq s}W_i)<n$ and we apply the induction hypothesis to $\pi^{-1}\big(\R^n \sm \bigcup_{i\leq s}W_i\big)\cap Z$ to get the result.  
\end{proof}

\begin{claim}\label{cl1}
Let $\{Y_x:\; x\in A\}$ be a definable family of sets of dimension $s<m$ such that for every $x\in A$  $X_x\subseteq Y_x$. If there is a uniform decomposition of the $Y_x$'s into $\cal R$-EM then there is such a decomposition for the $X_x$'s. 
\end{claim}
\begin{proof}Let $\{Y_{x,t}:\, t\in S_x\}$ be an \Cal R-decomposition of $Y_x$. Since all the $Y_{x,t}$'s are cells, they are uniformly homeomorphic to some open set via some projection. Thus by applying the induction hypothesis to $Y_{x,t}\cap X_x$ we get some \cal R-decomposition  $A_{x,t}=\{A_{x,t,t'}:\; t'\in S_{x,t}\}$ of $Y_{x,t}\cap X_x$. Thus $\{A_{x,t,t'}:\; t\in S_x,\, t'\in S_{x,t}\}$ is an \Cal R-decomposition of $X_x$ and we have the result by Claim \ref{cl10}. 
\end{proof}
\smallskip
 \par  {\em \textbf{\underline{Case \textbf{$n<m-1$}:}}}
 \medskip
 \par  Let $\pi$ be the projection on the first $m-1$ coordinates. By the induction hypothesis, there is  a decomposition of $\pi(X_x)$ into $\cal R$-EM. Thus we may assume that $\pi(X_x)$ is an \cal R-EM of dimension $n$ of the form $\bigcup_{t\in S_x}Y_{x,t}$. 
 It is easy to see that $\bigcup_{t\in S_x}\pi^{-1}(Y_{x,t})$ is an $\cal R$-EM of dimension $n+1<m$ and that it contains $X_x$. We apply Claim \ref{cl1} to get the result.
 \bigskip 

 \par {\em \textbf{\underline{Case $n=m-1$:}}} 
 \medskip
\par Since for every $x\in A$, $y\in \R^n$,  $\pi^{-1}(y)\cap X_x$ has dimension $0$, By Theorem \ref{FM} there is a function $f:\R^{p+n+k}\to \R$, definable in \Cal R, such that $$\pi^{-1}(y)\cap X_x\subseteq \cl{f(x,y,P^k)}.$$
By o-minimality,  we may assume that  $f$ is continuous (thus we do not have anymore that $\dom(f_t)=\R^n$). 
\par \textbf{From now, to simplify the notations we do not mention the $x$ in index for $x\in A$. Keep in mind that the following construction can be made uniformly for $x\in A$.}  
 \par For every  $t\in P^k$  let $Y_{t}=\Gamma(f_t )$ and  $Y=\cl{\bigcup_{t\in P^k}Y_{t}}$. Note that by Lemma \ref{DIM}, $\dim(Y)=n$ and by Claim \ref{cl1}, it is sufficient to prove that $Y$ admits a decomposition into \cal R-EM.


Thus, it is sufficient to prove that $Y$ admits an \cal R-decomposition. Let $Z=\bigcup_{t\in S}Y_t$, by Claim \ref{cl10}, there is a decomposition of $Z$ into finitely many \Cal R-EM. Thus, we may assume that $Z$ is an \Cal R-EM and let $\{W_t:\, t\in S'\}$ be a small subfamily of a family definable in \Cal R such that for every $t\in S$ there is $t'\in S'$, $\pi(Y_t)=W_{t'}$.  Moreover, by Fact \ref{pi good}, we may assume that $Y\cap \pi^{-1}(\bigcup_{t\in S'}W_t)$ has dimension $0$.
\begin{claim} \label{cl 00}
Let $A,B\subseteq Z$ be two graphs of continuous functions $g,h:W_t\to \R$ (for some $t\in S'$). If there is  $x\in W_t$ such that $g(x)<h(x)$ then for every $x\in W_t$ $g(x)<h(x)$. 
\end{claim}
\begin{proof}
We assume that it is not the case.
By the Intermediate Value Theorem, there is $x\in Z$ such that $g(x)=h(x)$ and $A\cap B\neq \emptyset$. 
That contradicts the definition of an \cal R-EM.
 \end{proof}
 \textbf{From now, in order to simplify the notations, we fix $W=W_t$ for some $t\in S'$.}
\begin{notation}\label{A<B}
In the setting of Claim \ref{cl 00}, if there are $A,B\subseteq Z$, two graphs of continuous functions over $W $  such that there is $z\in W$ with $\pi_m(\pi^{-1}(z)\cap A )<\pi_m(\pi^{-1}(z)\cap B)$ then this is true for every $z\in W$ and we denote this fact by $A<B$.
\end{notation}


\par we need now to split the proof in two cases: one for expansions of the real field and an other for \Cal R semibounded. In both cases we prove that there is small subfamily $\{A_t:\, t\in S'\}$ of a family of cells of dimension $n$ definable in  \Cal R such that $X\subseteq \bigcup_{t\in S'}A_t$. Thus applying Claim \ref{cl10}, we get the result.
\medskip
\par {\em \textbf{Case $1$: \cal R expands the real field:}} 
\\From now on, by replacing $Y$ by $\rho(Y)$, where $\rho$ is the homeomorphism  $$\rho: \R^{m}\to \R^m $$ 
$$\rho:x\mapsto \big( x_i/(1+|x_i|)\big)_i,$$ we may assume that $Y$ is bounded. Note that $\rho$ preserves the \Cal R-EM and the fact that for every $x\in \pi(Y)$, $\dim(\pi^{-1}(x)\cap Y)=0$.
\begin{claim}\label{cl 5}
Let $(Y_{t_i})_{i}$ be a decreasing sequence (of graphs of functions over $W$) and let $A$ be its pointwise limit (which exists because $Y$ is bounded). Then:\begin{enumerate}
\item $A$ is the graph of a continuous function,
\item $\cl{\bigcup_{i}Y_{t_i}}\cap \pi^{-1}(W)\sm \bigcup_i Y_{t_i}=A$,
\item $\lim_H(\cl{Y_{t_i}})=\cl{A}$.
\end{enumerate} 
\end{claim}
\begin{proof}
For the first property, let $x\in W$ and let $a,b\in \cl{A}\cap \pi^{-1}(x)$. We show that $a=b$. Since  $(Y_{t_i})_i$ is a decreasing sequence of graphs of continuous functions, by the Intermediate Value Property, $[a,b]\subseteq \cl{\bigcup_{i}Y_{t_i}}\subseteq Y$. If $a\neq b$, that is a contradiction with  the fact that $\dim(\pi^{-1}(x)\cap Y)=0$. The second property is obvious and the third is just a consequence of Corollary \ref{Hlim=cl}. 
\end{proof}

\par Let $$H=\cl{\{\cl{Y_{t}\cap \pi^{-1}(W)}:\; t\in \R^k\}}^H.$$ By Fact \ref{Hausdorff fam}, $H$ is a family definable in $\cal R$, thus it has the form $\{X_a:\; a\in B\}$. Let $$H_P=\{a\in B:\; X_a\cap \pi^{-1}(W)\subseteq Y \text{ and is the graph of a continuous function over $Z$}\}.$$  
Since, by Claim \ref{cl 5} for $a\in H_P$, $X_a\cap \pi^{-1}(W)$ is uniquely determined by $\pi^{-1}(z)\cap X_a$ for any $z\in W$ and that by assumption, $\pi^{-1}(z) $ has dimension $0$, we may assume that $H_P$ has dimension $0$.
 Thus, we have  $$Y\cap \pi^{-1}(W)=\bigcup_{a\in H_P}X_a\cap \pi^{-1}(W)$$
 and we have the result for $\pi^{-1}(W)\cap Y$. 
\bigskip 
 \par It is easy to see that we can do this uniformly   $t\in S'$ and $Y\cap \bigcup_{t\in S'}\pi^{-1}(W_t)$ has a decomposition into $\cal R$-EM.  

\bigskip
\par {\em \textbf{Case $2$: \cal R is semibounded:}}

  \smallskip
\par  By Proposition \ref{lm sb},we may assume that  there are a vector subspace $V$, and a family $\{B_t:\; t\in A'\}$ such that:
\begin{enumerate}
\item $\{|B_t|:\; t\in A'\}$ is bounded,
\item for every $t\in A'$, $Y_t+V=B_t+V$,
\item $\dim(Y_t+V)=\dim(Y_t)=n$,
\item For every box $C\subseteq \pi(Y_t)$, $\pi^{-1}(C)\cap Y_t+V\subseteq Y_t$.
\end{enumerate}  
\begin{claim}\label{cl000}
There is a family $\{B'_t:\; t\in A'\}$ such that: 
\begin{enumerate}
\item for every $t\in A'$ $\dim(B'_t)=\dim Y_t=n$,
\item there is a bound $\rho$ on $\{|B'_t|:\; t\in A'\}$,
\item for every $t,t'\in A'$ if $\pi(Y_t)=\pi(Y_{t'})$ then $\pi(B'_t)=\pi(B'_{t'})$.
\item for every $t\in A'$, $Y_t+V$ is the graph of a continuous function over $\pi(Y_t+V)$. 

\end{enumerate}
\end{claim} 
\begin{proof}Let $\rho'$ be the bound on $\{|B_t|:\; t\in A'\}$.
\par By definable choice in o-minimal structures, there is a family $\{x_t:\; t\in A\}$ such that $x_t\in W_{t}$. For $t\in A$, let $B''_t=\cal B(x_t,2\rho')$. For $t'\in A'$ we denote by $\pi(t')$, $t\in A$ such that $\pi(Y_{t'})=W_{t}$. let $$B'_{t'}=\pi^{-1}(B''_{\pi(t')})\cap Y_{t'}+V.$$
\par We show that $\{B'_t:\; t\in A'\}$ satisfies all the properties we need. The first property is obvious. For the second property, we first observe that since $\dim(Y_t)=\dim(Y_t+V)=n$, $\pi(V)\simeq V$. 
 Let $t\in A'$, let  $x,y\in B'_t$. By Proposition \ref{lm sb}, $\pi^{-1}(x_{\pi(t)})\cap Y_t$ belongs to some translate $B_t+v$ for some $v\in V$.  There are also $v_x,v_y\in V$, $b_x,b_y\in B_t+v$ such that $x=b_x+v_x$ and $y=b_y+v_y$. Thus $$|x-y|\leq |b_x-b_y|+|v_x-v_y|.$$
 Since $$|b_x-b_y|<\rho'$$ and that $$v_x-v_y\in V\cap \pi(-1)(\cal B(0,2\rho'))$$ (which is bounded since $V\simeq \pi(V)$), there is a bound on $\{|B'_t|:\; t\in A'\}$.
 \medskip
 \par The third property is obvious. For the fourth, since $\pi^{-1}(Y_t)\cap Y_t+V=Y_t$, $\pi(B_t)\simeq B_t$. By proposition \ref{lm sb} (4), $Y_t+V$ is the disjoint union of translates of $B_t$ by elements of $V$ there are no $x,y\in Y_t+V$ such that $\pi(x)=\pi(y)$. Thus, since $Y_t+V$ is connected (since $Y_t$ and $V$ are) we have the result.
\end{proof}
For every $t\in S$, $\pi^{-1}(W_{t})\cap Z$ is composed of connected components which are graphs of continuous functions over $W_{t}$. Thus we introduce the notation following Claim \ref{cl 00}. For $Y_{t}, Y_{t'}$ such that $\pi(Y_{t})=\pi(Y_{t'})$. We denote by $Y_{t}<Y_{t'}$ if there is $z\in\pi(Y_{t})$ such that $\pi_m(\pi^{-1}(z)\cap Y_{t})<\pi_m(\pi^{-1}(z)\cap Y_{t'})$ (and by Claim \ref{cl 00} of Theorem \ref{L-DEM} it is true for every $z\in \pi(Y_{t})$). 

 \begin{claim}\label{cl111}
 Let $(t_n)_n$ be a sequence in $S'$ such that there is $t\in S $ with  $\pi(Y_{t_n})=W_{t}$ for every $n$, $(Y_{t_n})_n$ is a decreasing sequence such that there is $a\in \cl{\bigcup_n Y_{t_n}}\sm \bigcup_nY_{t_n}$. Then 
 \begin{enumerate}
 \item $(B'_{t_n})_n$ converges pointwise to $A$, the graph of a continuous function over $\pi(B'_{t_n})$ (for any $n$).
 
 \item $Y_{t_n}+V\to_nA+V$ pointwise and $A+V$ is the graph of a continuous function. 
 \end{enumerate}
 \end{claim}
 \begin{proof}
 First of all, by translation, we may assume that $\pi(a)\in \pi(B'_{t_n})$. Since there is a uniform bound $\delta$ on the $B'_{t_n}$, it is easy to see that $A\subseteq \cal B(a,\delta)$. Since $(B'_{t_n})_n$ is a decreasing sequence that is bounded from below, $\bigcup_{n}B'_{t_n}$ is contained in a bounded set and we may apply Claim \ref{cl 5} to get the first part. 
 \par The second part is coming from the fact that  for any $x\in W +\pi(V)$, there is $c\in V$ and an open neighbrohood $x\in C\subseteq W+\pi(V)$ such that    for every $n$:
$$\pi^{-1}(C)\cap Y_{t_n}+V\subseteq c+B_{t_n}.$$
Thus, we have the result by translation.
  \end{proof}
 As in the case where $\cal R$ expands the real field, it is sufficient to prove that the family of pointwise limits of $\{Y_t:\; t\in S'\}$ is definable. Moreover, by Claim \ref{cl000}.(4), $$B'_t+V\cap \pi^{-1}(W_{\pi(t))})=Y_t.$$ Thus  it is sufficient to prove that the family of Hausdorff limits of $\{\cl{B'_t}:\; t\in S'\}$ is definable.

\par For $t\in A'$, $a\in \R^m$, let $E_{t,a}=B_{t}\cap \cal B(a,3\rho)$. Let 
$$\cal W=\{\cl{E_{t,a}-a}:\; t\in A',a\in \R^m\}.$$
It is obvious that $\cal W$ is a semibounded family of compact sets and that all its members are contained in $\cal B(3\rho)$.    Let $H=\{X_t:\; t\in F\}$ be the set of Hausdorff limits of $\cal W$ (that is definable in \Cal R). Let  
 $$\cal A=\{(X_t\cap \pi^{-1}(B''_{t'}-\pi(a))+a+V\cap \pi^{-1}(W_{t'}):\; t'\in A,t\in F,a\in \R^m\}$$
 and let $$\cal A'=\{(X_t\cap \pi^{-1}(B''_{t'})-\pi(a))+a +V\cap \pi^{-1}(W_{t'}):\; t\in S',t\in F,a\in \R^m $$ $$\text{ such that }X_t+a+V\cap \pi^{-1}(W_{t'})\subseteq Y\}.$$
 It is easy to see that $\cal A'$ is a small subfamily of $\cal A$ and that $Y\cap \pi^{-1}(W_t)=\bigcup \cal A'$ and we have the result. As in the case where \Cal R expands the real field, it is easy to see that we can do this construction uniformly for $t\in S'$ and we have the result.
\bigskip
\par \textbf{\underline{\em case $m=n$:}}
\smallskip
\par We assume that $X=\intr(X)$. By induction, there is a decomposition of $\bd(X)$ into $\cal R$-EM. Let  $ s\in \N$ such that for $i\leq s $,  $\bigcup_{t\in S^i} Z^i_{t}$ is an $\cal R$-EM of dimension $m-1$ coming from the  family $\{Z^i_t:\; t\in A_i\}$ that is definable in \Cal R and $$\dim\bigg(\bd(X)\setminus \bigcup_i Z^i\bigg)<n-1.$$
We first need the following claim:
\begin{claim}
Let $Z$ be a definable set of dimension $s<m$. Then there is a decomposition of $Z$ into finitely many \cal R-EM $Z_1,\ldots, Z_a$ such that, for $\{Z_{i,t}:\, t\in S_i\}$ the \cal R-decomposition of $Z_i$ and $\pi$ the projection on the first $m-1$-coordinates: $$(*):\; \text{for every } i,j,t\in S_i,t'\in S_j\, \pi(Z_{i,t})\cap \pi(Z_{j,t'})\text{ either is empty or is equal.}$$  
\end{claim}
\begin{proof}
Let $W=\cl{\bigcup_{i,t\in S_i}\bd(\pi(Z_{i,t}))}$. Observe that $W$ has dimension lower than $m-1$. Let $W_1,\ldots,W_l$ be some decomposition into \Cal R-EM of $\R^{m-1}\sm W$. Then for $\{W_{i,t}:\; t\in S'_i\}$ an \Cal R-decomposition of $W_i$, the family $\{\pi^{-1}(W_{i,t}\cap Z_{j,t'}):i,j,t\in S'_i,t'\in S_j\}$ satisfies $(*)$. We deal with $\pi^{-1}(W)\cap Z$ easily by induction.
\end{proof}
Thus, we may assume that the $Z_i$'s satisfies $(*)$  and let $W_1,\ldots, W_j$ be an \cal R-decomposition of $\R^{m-1}$ that witnesses it and we may assume that $\bigcup_i W_i=\R^{m-1}$.
 Moreover, since $\bd(X)$ is closed, $$\pi^{-1}(\pi(Z^i_t))\cap \bd(X)$$ is relatively closed in $$\pi^{-1}(\pi(Z^i_t)).$$  
\par  Let $B'$ be the union of the $W_i$ of dimension lower than $m-1$. and let $B=\pi^{-1}(\cl{B'})\cap \bd(X)$.  Note that $B$ has dimension lower than $m$. We define $\phi:\R^m\setminus \big(\bd(X)\cup B\big)\to \R$ as follow:
$$\phi(y)=\sup\bigg\{\pi_m\big(\{\pi(y)\}\times(-\infty, y_m)\cap \bd(X)\big)\bigg\}$$
and $\delta:\R^m\setminus \big(\bd(X)\cup B\big)\to \R$ as follow:
$$y \mapsto s \in S_i\text{ such that $Z^i_{s}$ contains $\phi (y)$}.$$
Note that $\phi$ is well defined because $\bd(X)$ is relatively closed in $\bigcup_{i\leq s}\pi^{-1}(\pi(Z^i))$.
\par For $y\in \R^m\setminus \big(\bd(X)\cup B\big)$, we  define:$$\Delta(y)=\delta^{-1}(\delta (y)).$$

\par We note that $ \bd(X) \cup B$ has lower dimension than $m$ and that it contains $\bd(X)$. Therefore, if $y\in X$ then $\Delta(y)\subseteq X $. Moreover, since $\Delta(y)$ only depends on $\delta(y)$, the family $$\Big\{\Delta(y):\; y\in  \R^m\setminus  \bigcup_i \bigg(\bd(X)\cup B\bigg)\Big\}$$ is small. 

It is also a  subfamily of  $$\Big\{\{y\in \R^m:\; y\in \intr\big([Z^i_{t_{1}},Z^j_{t_{2}}]\big)\}:\; t_{1}\in A_i, t_{2}\in A_j\Big\}$$
which is obviously definable in \Cal R. 
\par Since  $ \bigcup_i \bigg(\cl{Z^i}\cup B^i\bigg)$ has lower dimension than $m$, we apply the induction hypothesis to get the result. 
\end{proof}

\begin{corollary}
We assume that $\WR$ is d-minimal. Let $X$ be a definable set then there is a finite decomposition of $X$ into some \cal R-EM $X_1,\ldots,X_n$ such that for every $i$ there is some \cal R-decomposition $\{X_t:\; t\in S\}$ of  $X_i$ and for every $t\in S$, $X_t$ is relatively open in $X_i$
 \end{corollary}
 \begin{proof}
 By d-minimality, we may assume that $X$ is an EM such that for every $x\in \pi(X)$ (for $\pi$ some projection onto $\R^{\dim(X)}$), $\pi^{-1}(x)\cap X$ is discrete. The result follows directly.
 \end{proof}
\section{Discussion and questions}
 In this section we first discuss some consequences of Theorem \ref{L-DEM}. Then we present some related questions and discuss their likelyness. 

 \subsection{No new smooth functions in the d-minimal setting}
 Theorem \ref{main2} answers to a question asked in \cite{ES}. As a consequence of \cite{ES} and \cite{ES2} we get the following theorem:
 \begin{theorem}
 Let $\WR$ be d-minimal. If $f$ is a smooth  function definable in $\WR^\#$, over a domain that is definable in \Cal R then $f$ is definable in $\cal R$. 
 \end{theorem}

 \subsection{NIP}
We just say a word on NIP in the $P$-internal  setting. We thank Erik Walsberg for pointing this out.
\begin{defn}
We say that $\WR$ is $P$-internal if for any small set $S\subseteq \R^n$ there is a function, definable in \Cal R $f:\R^m\to  \R^n$ such that $S\subseteq f(P^m)$. 
Let $P_{ind}$ be the structure induced on $P$ by $\WR$.
 Let $\la\cal A, B\ra$  be a structure such that the language of $\cal A$ is $\cal R'$ and $B\subseteq A$ where $A$ is the underlying set of \Cal A. A formula is said to be $B$-bounded if it has the form $$Q_0x_0\in B,\ldots, Q_n x_n\in B\, \phi (x_0,\ldots,x_n)$$ where $Q_i\in \{\exists,\forall\}$ and $\phi$ is an $\cal R'$-formula.
\par We say that the theory of $\la\cal A, B\ra$ is bounded if every formula is $B$-bounded.
\end{defn} 
\medskip

\begin{proposition}\label{P(A) def}
Let $\WR$ be i-minimal and $P$-internal. Then $\WR$ is NIP if and only if $P_{ind}$ is.  
\end{proposition}

\begin{proof}
The left-to-right direction is obvious. For the right-to-left direction, it is just an application of   Chernikov and Simon's \cite[Corollary 2.5]{SiCher}  that states that  if $Th(\cal R)$ and $P_{ind}$ are NIP and $Th(\la\Cal R,P\ra)$ is bounded then $Th(\la\cal R,P\ra)$ is NIP. The boundedness being fulfilled by \cal R-DEM and by $P$-internality, we have the result. 
\end{proof}

\subsection{Questions}
All the known d-minimal structures are reducts of structures of the form $\la\cal R,P\ra^\#$  (for \cal R, an o-minimal structures and $P \subseteq \R$ a set of dimension $0$. Thus  they satisfy some weak form of Theorem \ref{main2}. It is natural to ask wether every d-minimal structures has this form or not. 
 \begin{question}\label{Q1}
   For $\cal M$, a d-minimal expansion of $\la\R,<\ra$, is there an o-minimal structure $\cal M'$ and some sequence of small sets $(S_i)_{i\in I}$ such that $\cal M\subseteq \la\cal M',(S_i)_{i\in I}\ra$  and such that the last structure is d-minimal? 
   \end{question}
  The above question is quite naive and the answer is probably negative, but it could be the starting point of an interesting sequence of questions.
  One way to start looking for a counterexample to Question \ref{Q1} is to look at the following type of structures: Let $\la\Cal R,f:\R^2\to \R\ra$ be o-minimal and let $P\subseteq \R$ be a set of dimension $0$ such that $\la\Cal R,P\ra$ is d-minimal and $\la\cal R,f,P\ra$ is not. 
  \begin{question}\label{Q2}
  Is there such $\Cal R$, $f$ and $P$ such that $\la\cal R,f\upharpoonright P\times \R\ra$ is d-minimal?
  \end{question}
\par We finish by taking a look at two examples around Question \ref{Q2}. The structure $$\la\R,<,+,\big(f:\Z\times\R\to \R,(t,x)\mapsto tx\big)\ra$$ is not d-minimal. Take:   $$\bigcup_{t\in \Z}f_t^{-1}(\Z)=\Q.$$ Moreover we show that the multiplication is definable on $\Q$. First, the multiplication is definable on $\Z^2$: take $(x,y)\mapsto f_x(y)$. For $(x,y), (z,t)\in \Z^2$ let $(x,y)*(z,t)$ be the multiplication coordinate by coordinate. Let  $g:(x,y)\in \Z^2\mapsto f_y^{-1}(x)$. Let $h:z\in \Q\mapsto (x,y)\in \Z^2$ such that $y$ is the smallest element of $\N$ with  $g(x,y)=z$.  We define the  multiplication on $\Q^2$ by $$(x,y)\mapsto g(h(z)*h(z')).$$ 
\par Therefore the full multiplication on $\R^2$ is definable by taking the closure of the graph of the multiplication on $\Q^2$ and thus $$\la\R,<,+,\big(f:\Z\times\R\to \R,(t,x)\mapsto tx\big)\ra$$ defines the whole projective hierarchy (see for example \cite{kechris}).
 \smallskip
 \par  An other example in the same taste is $$\la\cl{\R},\big(f:2^\Z\times\R\to\R,(t,x)\mapsto x^t\big)\ra$$
  which defines a dense co-dense set  by taking $$\bigcup_{t\in 2^\Z}f_t^{-1}(2^\Z)=\{2^{n/2^m}:\; n,m\in \Z\}.$$ Moreover, it is easy to see that this set is $\omega$-orderable (see \cite{HW} the definition right after Theorem A in the introduction) and thus $$\la\cl{\R},\big(f:2^\Z\times\R\to\R,(t,x)\mapsto x^t\big)\ra$$ defines the whole projective hierarchy (see \cite{HW} Fact 1.2). 

\end{document}